\documentclass[12pt]{jloganal}
\usepackage{amsmath, amsthm, amsfonts, amssymb}
\usepackage{mathrsfs}
\usepackage{latexsym}
\usepackage[all]{xy}
\usepackage{color}
\usepackage{verbatim}
\usepackage{mathtools}

\newtheorem{theorem}{Theorem}[section]
\newtheorem{lemma}[theorem]{Lemma}
\newtheorem{proposition}[theorem]{Proposition}
\newtheorem{corollary}[theorem]{Corollary}

\theoremstyle{remark}
\newtheorem{remark}[theorem]{Remark}

\theoremstyle{definition}
\newtheorem{definition}[theorem]{Definition}

\newtheorem{rmks}{Remarks}[section]

\newtheorem*{Examples}{Examples}

\newcommand{\bd}[1]{\boldsymbol{#1}}

\newcommand{\A}{\mathcal{A}}
\newcommand{\B}{\mathscr{B}}
\renewcommand{\C}{\mathscr{C}}

\newcommand{\N}{\mathbb{N}}

\renewcommand{\Pr}{\mathbb{P}}
\renewcommand{\Q}{\mathbb{Q}}
\renewcommand{\R}{\mathbb{R}}

\renewcommand{\Z}{\mathbb{Z}}

\newcommand{\shift}[2]{#1^{\uparrow #2}}

\renewcommand{\epsilon}{\varepsilon}

\newcommand{\ML}{Martin-L\"{o}f }

\newcommand{\Fouche}{\text{Fouch\'{e} }}
\newcommand{\Morters}{\text{M\"{o}rters }}

\newcommand{\uh}{\upharpoonright}

\newcommand{\floor}[1]{\lfloor #1 \rfloor}

\newcommand{\cdim}{\mathrm{cdim}}

\newcommand{\indic}[1]{\bd{1}_{#1}}

\definecolor{purple}{rgb}{.7,0,.7}

\usepackage{xcolor}	
	\usepackage{soul}

	\definecolor{lightblue}{rgb}{.60,.60,1}

\newtheorem{question}{Question}

\title{On zeros of Martin-L\"of random Brownian motion}

\author{Kelty Allen}
\givenname{Kelty}
\surname{Allen}
\address{The University of California, Berkeley, Department of Mathematics, 719 Evans Hall \#1049, Berkeley, CA 94720-3840 USA }
\email{kelty@math.berkeley.edu}
\urladdr{math.berkeley.edu/~kelty/}

\author{Laurent Bienvenu}
\givenname{Laurent}
\surname{Bienvenu}
\address{Laboratoire CNRS J.-V. Poncelet, 119002, Bolshoy Vlasyevskiy Pereulok 11, Moscow, Russia}
\email{laurent.bienvenu@computability.fr}
\urladdr{www.liafa.jussieu.fr/~lbienven/}

\author{Theodore A. Slaman}
\givenname{Theodore A.}
\surname{Slaman}
\address{The University of California, Berkeley, Department of Mathematics, 719 Evans Hall \#3840, Berkeley, CA 94720-3840 USA }
\email{slaman@math.berkeley.edu}
\urladdr{math.berkeley.edu/~slaman/}

\keywords{Algorithmic randomness, Brownian motion, computable analysis}
\subject{primary}{msc2010}{03D32, 60J65}

\arxivreference{}
\arxivpassword{}

\volumenumber{}
\issuenumber{}
\publicationyear{}
\papernumber{}
\startpage{}
\endpage{}
\doi{}
\MR{}
\Zbl{}
\received{}
\revised{}
\accepted{}
\published{}
\publishedonline{}
\proposed{}
\seconded{}
\corresponding{}
\editor{}
\version{}

\begin{document}

\begin{abstract}
We investigate the sample path properties of  \ML random Brownian motion. We show (1) that many classical results which are known to hold almost surely hold for every \ML random Brownian path, (2) that the effective dimension of zeroes of a \ML random Brownian path must be at least 1/2, and conversely that every real with effective dimension greater than 1/2 must be a zero of some \ML random Brownian path, and (3) we will demonstrate a new proof that the solution to the Dirichlet problem in the plane is computable. 
\end{abstract}

\maketitle



\section{Background and notation}\label{sec:background}

\subsection{Brownian motion}

Heuristically, Brownian motion is the random continuous function resulting from the limit of discrete random walks as the time interval approaches zero. The paths of Brownian motion are considered typical with respect to \emph{Wiener measure} on a function space, generally $C[0, 1]$, $C[0, \infty)$, or $C[I, \R^n]$ for $I = [0, 1]$ or $[0, \infty)$ The \ML random elements of a function space with respect to Wiener measure are known as \ML random Brownian motion. Fouch\'{e} showed that the class of \ML random Brownian motion is the same as the class of complex oscillations, a class of functions defined by Asarin and Pokrovskii \cite{AP} and later investigated to a greater degree by Fouch\'{e} \cite{FoucheArithmeticalRepresentation,FoucheDescriptiveComplexity,FoucheDynamics}, Davie and Fouch\'e~\cite{FoucheDavie},  Kjos-Hanssen, Nerode~\cite{KHN}, and Szabados~\cite{KHS}.

In this article, we continue the study of \ML random Brownian motion. We will demonstrate that many classical theorems which hold almost surely hold for every \ML random Brownian path, we will prove results toward a classification of the effective dimension of the zeroes of \ML random Brownian motion, and we will demonstrate a new proof that the solution to the Dirichlet problem in the plane is computable. 

We will use $2^{\omega}$ to denote infinite binary strings, which we will sometimes identify with reals on [0, 1]. We denote the space of continuous functions $f: [0, 1] \to \R$ and $f: \R^{\geq 0} \to \R$ by $C[0, 1]$ and $C[\R^{\geq 0}]$ respectively. For other cases, the space of continuous functions from a set $X$ to a set $Y$ will be denoted by $C(X, Y)$.

\emph{Standard (one dimensional) Brownian motion} is a real-valued stochastic process $\{\B(t): t \in I\}$ ($I=[0,1]$ or $I=[0,\infty)$) where the following hold. First, for any  $t_0 < t_1< ... < t_n$, the increments $\B(t_n) - \B(t_{n-1}), \B(t_{n-1}) - \B(t_{n-2}), ... , \B(t_2) - \B(t_1)$ are independent random variables. Second, for all $t\geq 0$ and $h > 0$, the increments $\B(t+h) - \B(t)$ are normally distributed with mean $0$ and variance~$h$. Third, $\B(0) = 0$ almost surely, and~$\B$ is almost surely continuous. These requirements induce a measure on a function space called Wiener measure, and which we will denote by $\Pr$. The values taken by the random variable~$\B$ are called \emph{sample paths}, or simply \emph{paths}.

It is possible to define Brownian motion starting at any point $x$ at time~$0$, rather than starting at the origin, in which case we will denote the corresponding measure by $\Pr_x$ (in other words, $\Pr_x(\B \in \mathcal{A}) = \Pr(x+\B \in \mathcal{A})$). When we wish to emphasize that we are talking about standard Brownian motion, we will use $\Pr_0$.  \\

We assume that the reader is familiar with algorithmic randomness and Kolmogorov complexity for binary sequences. One can consult the two books~\cite{DH,Nies} for a good overview of the subject. Furthermore, we assume some familiarity with Martin-L\"of randomness for computable probability spaces. G\'acs' lecture notes~\cite{Gacs-notes} and the two papers~\cite{HRAppofMLR,HRAppofEffProb} by Hoyrup and Rojas are the standard references on the subject. Our main reference for the classical theory of Brownian motion is the recent book by Morters and Peres~\cite{MP}.

\subsection{Effective aspects of Brownian motion}

The construction presented here is the Franklin-Wiener series representation of Brownian motion as found in~\cite{Kahane}. 

Let $\Delta_0(t)$ be the linear interpolation between the points $(0,0)$ and $(1,1)$. $\Delta_1(t)$ is the linear interpolation between points (0,0), (1/2,1/2), 
and (1,0). $\Delta_{i,j}(t)$ ($0<j < 2^i$) is the function  that linearly interpolates between $(j/2^i,0)$, $((j+1/2)/2^i, 2^{-j/2-1})$, and $((j+1/2)/2^i,0)$ and is equal to 0 everywhere else.

Classically, it is known that if $\xi_0, \xi_1, \{\xi_{i, j}\}_{i\in \N, j<2^i}$ are independent random variables following a normal distribution $\mathcal{N}(0,1)$, then the random variable~$\B$ defined by 
\[
\B(t) = \xi_{0} \Delta_{0}(t)  + \xi_{1} \Delta_{1}(t)  + \sum_{i} \sum_{j < 2^i} \xi_{i,j} \Delta_{i,j}(t)  
\]
is a standard Brownian motion on~$[0,1]$.

To extend Brownian motion to $C[\R^{\geq 0}]$, let $\{ \B_n(t) \}_{n \in \N}$ be independent Brownian motions on $C[0, 1]$. Then
\begin{equation}\label{eq:infinite-time}
\B(t) = \B_{\floor{t}}(t - \floor{t}) + \displaystyle \sum_{0 \leq i < \floor{t}} \B_i(1)
\end{equation}
satisfies the definition of Brownian motion for the space of continuous functions $C(\R^{\geq 0},\R)$.

In order to define Martin-L\"of randomness for Brownian motion, one needs to make sure that the space of continuous functions~$C[0,1]$ endowed with distance
\[
d(f,g) = ||f-g||_\infty
\]
and Wiener measure (denoted $\Pr$) is a computable probability space. 

The computability of $(C[0,1],\Pr)$ was proven by Fouch\'e and Davie~\cite{FoucheDavie,FoucheDescriptiveComplexity} (see next subsection for more details). One can take for dense set of points the piecewise linear functions which interpolate between finitely many points of rational coordinates, and for $p$ such a function and $r>0$ a rational number, the $\Pr$-measure of 
\[
\{f \mid ||f-p||_\infty < r\}
\]
is computable uniformly in a code for~$p$. 

Therefore, it is possible to define Martin-L\"of randomness for Brownian motion in the usual way: the \ML random elements of $(C[0,1],\Pr)$ are those which do not belong to the universal Martin-L\"of test $\bigcap_n \mathcal{U}_n$. To stress the difference between Brownian motion as a stochastic process and Martin-L\"of randomness on the space $(C[0,1],\Pr)$, we will use the cursive letter $\B$ for the random variable taking values in $C[0,1]$ and distributed according to~$\Pr$, and use the letter $B$ for individual elements of~$C[0,1]$. Recall that we refer to elements $B \in C[0,1]$ as (sample) paths, and therefore we will only talk about \emph{\ML random paths}, and not \ML random Brownian motion.  \\

Note that all of the above can be adapted in a straightforward way to the space $C[0,\infty)$, which by the above correspondence~(\ref{eq:infinite-time}) can be identified with $\omega$ copies of $(C[0,1],\Pr)$, endowed with the product measure $\Pr^\omega$.

\subsection{Layerwise computability}

Throughout the paper, we will make extensive use of the notion of \emph{layerwise computability} developed by Hoyrup and Rojas~\cite{HRAppofMLR,HRAppofEffProb}. Layerwise computability is a form of uniform relativisation: In computability theory, we often say that an element $y$ is \emph{computable in~$y$} if $y$ can be computed given~$x$ as oracle. We say that an expression $F(x)$ is \emph{computable uniformly in~ $x$} if~$F$ is a computable function on the space to which~$x$ belongs. There are many examples of this in computable analysis: $x^2$ is computable uniformly in~$x \in [0,1]$, the integral $\int f$ is computable uniformly in $f \in C[0,1]$ (endowed with the $||.||_\infty$ norm), etc.

Layerwise computability is a slightly weaker form of uniformity. First of all when we say that an expression $F(x)$ is layerwise computable, we only ask that it is defined for $x$ \ML random on the computable probability space~$\mathbb{X}$ it belongs to (see~\cite{HRAppofMLR,HRAppofEffProb} for the definition of computable probability space). Moreover, we only require uniformity on each ``layer" of~$\mathbb{X}$, uniformly in~$n$. A layer is a set of type $\mathcal{K}_n$, where $\mathcal{K}_n$ is the complement of $\mathcal{U}_n$, the $n$-th level of a universal Martin-L\"of test over~$\mathbb{X}$. An interesting aspect of layers is that they always are effectively compact, even if the space~$\mathbb{X}$ itself is not compact. So formally, we say that $F(x)$ is computable layerwise in~$x$ if there exists a partial computable function $G(.,.)$ such that $G(x,n)=F(x)$ for all~$x \in \mathcal{K}_n$. 

Layerwise computability is a very powerful tool to study constructive versions of classical results in probability theory and measure theory (as we shall see in this paper!). Perhaps the most important result using layerwise computability is the so-called `randomness preservation theorem':

\begin{theorem}[\cite{HRAppofMLR,HRAppofEffProb}]\label{thm:layerwise-computability}
Let $(\mathbb{X},\mu)$ be a computable probability space and $F$ a layerwise computable function over~$\mathbb{X}$ taking values in a computable metric space~$\mathbb{Y}$. Then:
\begin{itemize}
\item[(i)] The push forward measure~$\nu$ defined over $\mathbb{Y}$ by $\nu(\mathcal{A})=\mu(F^{-1}(\mathcal{A}))$ is computable
\item[(ii)] If $x$ is $\mu$-\ML random, then $F(x)$ is~$\nu$-\ML random. 
\item[(iii)] For every~$y \in \mathbb{Y}$ which is $\nu$-\ML random, there is some $\mu$-ML random $x \in \mathbb{X}$ such that $F(x)=y$. 
\end{itemize}
\end{theorem}

This theorem can for example be used to prove that $C[0,1]$ with the $||.||_\infty$ norm and Wiener measure is a computable probability space (as alluded to in the previous subsection). Indeed, Fouche and Davie proved that the function $\Phi$ which maps a sequence of reals $\xi_0, \xi_1, \{\xi_{i, j}\}_{i\in \N, j<2^i}$ to the function 
\[
\B(t) = \xi_{0} \Delta_{0}(t)  + \xi_{1} \Delta_{1}(t)  + \sum_{i} \sum_{j < 2^i} \xi_{i,j} \Delta_{i,j}(t)  
\]
is layer wise computable from $\mathbb{X}$ to $(C[0,1], ||.||_\infty)$, where $\mathbb{X}$ is the space of sequences of real numbers where each coordinate is distributed according to the normal distribution~$\mathcal{N}(0,1)$ independently of the others. It is obvious that $\mathbb{X}$ is a computable probability space. Thus, by the above theorem, the measure induced by~$\Phi$ on $(C[0,1], ||.||_\infty)$, which we know to be Wiener measure, is a computable measure. \\

Another important result we will need in several occasions is that one can compute the integral of layerwise computable functions. 

\begin{theorem}[\cite{HRAppofMLR}]\label{thm:layerwise-integral}
Let $f$ be a bounded layerwise computable function defined on some computable probability space $(\mathbb{X},\mu)$. Then the integral
\[
\int_{x \in \mathbb{X}} f(x)\, d\mu(x)
\]
is computable uniformly in an index of~$f$ and a bound for it. 
\end{theorem}


\section{Basic properties of Martin-L\"of random paths}

We begin by showing that the main ``almost sure" properties of classical Brownian motion hold for Martin-L\"of random paths. 

\subsection{Scaling theorem}\label{subsec:scaling}

The classical scaling theorem states that the map $B(t) \mapsto \frac{1}{a}B(a^2 t)$ is a Wiener-measure-preserving map from $C[0,1]$ to $C[0,1]$ (or $C[0, \infty) \to C[0, \infty))$ see for example Lemma~1.7 in~\cite{MP}. For \ML random paths, we have the following. 

\begin{proposition}
Let $B$ be a \ML random path of $C[0,1]$ (resp.\ of $C[\R^{\geq 0}]$). Then $\frac{1}{a}B(a^2 t)$ is also a \ML random path of $C[0,1]$ (resp.\ of $C[\R^{\geq 0}]$)  whenever~$B$ is random relative to~$a$.
\end{proposition}

\begin{proof}
The map $B(t) \mapsto \frac{1}{a}B(a^2 t)$ is $a$-computable measure preserving, therefore it preserves Martin L\"{o}f randomness relative to~$a$ by Theorem \ref{thm:layerwise-computability} relativized to~$a$.
\end{proof}


\subsection{Constructive strong Markov property}

The strong Markov property of Brownian motion asserts the following. Let~$T$ be a stopping time, that is, a random variable in $[0,\infty]$ which is a function of~$\B$, and such that deciding whether $\{T \leq t\}$ depends only on $\B \uh [0,t]$ (the restriction of $B$ to the interval $[0,t]$). If~$T(\B)$ is almost surely finite, then the process $\widehat{\B}$ defined by $\widehat{\B}(t) = \B(T(\B)+t)-\B(T(\B))$ is a Brownian motion independent of $\B \uh [0,T(\B)]$. 

From its classical version, we can derive a constructive version of the strong Markov property which will be very useful in the sequel. 

\begin{proposition}
Let $T$ be a layerwise computable stopping time. Then the function
\[
B(t) \mapsto \widehat{B}(t) 
\]
where $\widehat{B}(t) = B(T(B)+t)-B(T(B))$ is layerwise computable and if $B$ Martin-L\"of random, then $\widehat{B}$ is Martin-L\"of random relative to~$B \uh [0,T(B)]$. 
\end{proposition}

\begin{proof}
Consider the product space $C[0,\infty) \times C[0,\infty)$ endowed with the product measure $W \times W$. Consider the map
\[
(B_1,B_2) \mapsto \left( (B_1 \uh [0,T(B_1)])^\frown B_2\, , \, \widehat{B_1} \right)
\]
from $C[0,\infty) \times C[0,\infty)$ into itself, where $(B_1 \uh [0,T(B_1)])^\frown B_2$ is the concatenation of~$B_1$ up to time $T(B_1)$ and then continued according to $B_2$:
\[
\Big( (B_1 \uh [0,T(B_1)])^\frown B_2 \Big)  (t)= \left\{ \begin{array}{ll} B_1(t)~ \text{if}~  t \leq T(B_1)\\ B_1(T(B_1))+B_2(t-T(B_1))~ \text{if}~ t \geq T(B_1) \end{array} \right.
\]
By the strong Markov property, this map is measure preserving and it is layerwise computable since $B_1 \mapsto T(B_1)$ is. Thus, if $(B_1,B_2)$ is \ML random, the pair $ \left( (B_1 \uh [0,T(B_1)])^\frown B_2\, , \, \widehat{B_1} \right)$ is also \ML random, and thus, by van Lambalgen's theorem, $\widehat{B_1}$ is random relative to $(B_1 \uh [0,T(B_1)])^\frown B_2$. Since $T$ is a stopping time, $(B_1 \uh [0,T(B_1)])^\frown B_2$ computes $T(B_1)$ and thus computes $B_1 \uh [0,T(B_1)]$. Therefore, $\widehat{B_1}$ is \ML random relative to $B_1 \uh [0,T(B_1)]$.
\end{proof}

\subsection{Continuity properties}

In his paper establishing many of the local properties of \ML random Brownian motion~\cite{FoucheDynamics}, \Fouche shows every \ML random Brownian motion obeys a modulus of continuity $\phi(h)$ such that

$$
\limsup_{h \to 0} \sup_{0\leq t \leq 1-h} \frac{|B(t+h) - B(t)|}{\phi(h)} \leq 1
$$
and 
\begin{equation}\label{eqn:ModOfCont}
\phi(h) = O\left(\sqrt{h \log(1/h)}\right)
\end{equation}

It is possible to extend this result with big-O notation to the particular constant ($\sqrt{2}$) from the classical result, and moreover, while the classical result demonstrates that the modulus of continuity holds for ``sufficiently small"~$h$, we will demonstrate that ``sufficiently small" is layerwise computable from a \ML random path.

\begin{proposition} \label{prop:variation-sqrt-h-log-h}
Let $B$ be a ML random Brownian motion. Then for all $c < \sqrt{2}$, for all~$h_0$, there exists $h<h_0$ such that
\[
|B(t+h)-B(t)| > c \sqrt{h \log(1/h)}
\]
\end{proposition}

\begin{proof}
For a large $n$ (to be specified later), split the interval $[0,1]$ into chunks of size~$e^{-n}$ (omitting the last bit). For each $0 \leq k < e^n$, consider the event 
\[
A_k: |B((k+1)e^{-n})-B(k e^{-n})| \geq c \sqrt{e^{-n} n}
\]
(i.e., what we want, with $h=e^{-n}$)

Note that the $A_k$ are independent by definition of Brownian motion and by time-translation invariance, all have the same probability. Let us estimate the probability of $A_0$, which is the event: $|B(e^{-n})-B(0)| \geq c \sqrt{e^{-n} n}$. By scaling, it is also equal to the probability of the event: $|B(1)-B(0)| \geq c \sqrt{n}$. By the estimate given in Peres-Morters (Lemma 12.9), we have 
\[
\Pr(A_0) \geq \frac{c \sqrt{n}}{c^2 n +1} e^{- c^2 n/2}
\]
so, by assumption on $c$, there exists an $\alpha < 1$ such that for almost all~$n$
\[
\Pr(A_0) \geq e^{-\alpha n}
\]
Since the $A_k$ are independent, 
\[
\Pr(\text{no $A_k$ happens}) \leq (1- e^{-\alpha n})^{e^n} \sim e^{-e^{(1-\alpha)n}}
\]
Thus for $n$ taken large enough, this can be made arbitrarily small. Moreover, notice that $c$ can be supposed to be computable, which makes the $A_k$ $\Pi^0_1$ classes, hence the event ``no $A_k$ happens" corresponds to a $\Sigma^0_1$ class. Thus, we have a Solovay test that any \ML random Brownian motion should pass, and for such a \ML random Brownian path, there are infinitely many~$n$ for which some $A_k$ happens. 
\end{proof}

\begin{proposition}\label{prop:modulus-of-continuity-2}
Let $B$ be a ML random Brownian motion. Then for all $c > \sqrt{2}$, there is~$h_0$, such that for all $h<h_0$ and all~$t$
\[
|B(t+h)-B(t)| < c \sqrt{h \log(1/h)}
\]

Moreover, $h_0$ is layerwise computable in $B$.
\end{proposition}

The proof is the same as that of \Morters and Peres Theorem 1.14 \cite{MP}, with the addition of keeping track of the layerwise computability of $h_0$. We recall the proof for completeness.

 We first look at increments over a class of intervals, which is chosen to be sparse, but big enough to approximate arbitrary intervals. More precisely, given $n, m \in \N$, we let $\Lambda_n(m)$ be the collection of all intervals of the form
$$
[(k-1+b)2^{-n+a}, \,\,\, (k+b)2^{-n+a}],
$$
for $k \in \{1, ... 2^n\}, \,\, a, \, b \in \{0, \frac{1}{m}, ... , \frac{m-1}{m}\}$. We further define $\Lambda (m) \coloneqq \bigcup_n \Lambda_n(m)$. 

\begin{lemma}\label{lemma:ModOfContLemma1}
For any fixed $m$ and $c > \sqrt{2}$, for $B(t)$ a \ML random Brownian motion, there exists $n_0 \in \N$, layerwise computable in $B(t)$, such that for any $n \geq n_0$, 
$$
|B(t) - B(s)| \leq c \sqrt{(t-s)\log\frac{1}{(t-s)}} \,\,\,\,\,\,\,\, \text{for all } [s, t] \in \Lambda_m(n).
$$
\end{lemma}
\begin{proof}
From the tail estimate for a standard normal variable $X$, see, for example \cite{MP} Lemma 12.9, we obtain

\begin{multline*}
\Pr\bigg\{ \sup_{k \in \{1, ... , 2^n\}} \sup_{a, b \in \{0, \frac{1}{m}, ... ,\frac{m-1}{m}\}} \\
 |B((k-1+b)2^{-n +a}) - B((k + b)2^{-n+a})| > c\sqrt{2^{-n+a}\log(2^{n+a})}\bigg\} 
\end{multline*}

$$
 \leq 2^nm^2\Pr\{X>c\sqrt{\log(2^n)}\}
 $$
 \begin{equation}\label{eqn:Bob}
 \leq \frac{m^2}{c\sqrt{\log(2^n)}}\frac{1}{\sqrt{2\pi}} 2^{n(1-\frac{c^2}{2})}.
\end{equation}

Note that $c$ can be taken to be computable, so for fixed $m, n \in \N$ the event 

\begin{multline*}
\sup_{k \in \{1, ... , 2^n\}} \sup_{a, b \in \{0, \frac{1}{m}, ... ,\frac{m-1}{m}\}} \\
|B((k-1+b)2^{-n +a}) - B((k + b)2^{-n+a})| > c\sqrt{2^{-n+a}\log(2^{n+a})}
\end{multline*}

is computable in $B(t)$ and the right hand side of \ref{eqn:Bob} is summable, giving a Solovay test which every \ML random Brownian motion $B(t)$ will pass. 

The standard proof of the equivalence of Solovay randomness and \ML randomness gives a uniform way of converting a Solovay test $\{\mathcal{S}_i\}$ to a \ML test $\{\mathcal{U}_j\}$. See, for example, \cite{DH}. Thus knowing a~$k$ such that a \ML random path $B(t) \not \in \mathcal{U}_k$ gives us an $n_0$ where the path no longer appears in any $\mathcal{S}_n$ for $n > n_0$.  Thus the $n_0$ given in the proof above is layerwise computable in~$B$. 

\end{proof}

\begin{lemma}\label{lemma:ModOfContLemma2}
Given $\epsilon > 0$ there exists $m \in \N$ such that for every interval $[s, t] \subset [0, 1]$ there exists an interval $[s', t'] \in \Lambda(m)$ with $|t - t'| < \epsilon (t - s)$ and $|s - s'| < \epsilon (t - s)$. 
\end{lemma}
\begin{proof}
See \cite{MP}, Lemma 1.17

\end{proof}

\begin{proof}[Proof of Proposition \ref{prop:modulus-of-continuity-2}]

Given $c > \sqrt{2}$, pick $0<\epsilon < 1$ small enough to ensure that $c^*\coloneqq c - \epsilon > \sqrt{2}$ and $m \in \N$ as in Lemma \ref{lemma:ModOfContLemma2}. Using Lemma~\ref{lemma:ModOfContLemma1} we choose $n_0 \in \N$ large enough that, for all $n \geq n_0$ and all intervals $[s', t'] \in \Lambda_n(m)$, almost surely,
$$
|B(t') - B(s')| \leq c^* \sqrt{(t' - s') \log\frac{1}{(t' - s')}}. 
$$
Now let $[s, t] \subset [0, 1]$ be arbitrary, with $t - s < \min(2^{-n_0}, \epsilon$), and pick $[s', t'] \in \Lambda(m)$ with $|t - t'| < \epsilon (t - s)$ and $|s - s'| < \epsilon (t - s)$. Then, recalling \ref{eqn:ModOfCont}, there is a $C$ such that 
\begin{gather*}
|B(t) - B(s)| \leq |B(t) - B(t')| + |B(t') - B(s')| + |B(s') - B(s)|\\
\leq C \sqrt{|t - t'| \log\frac{1}{|t - t'|}} + c^* \sqrt{(t' - s')\log\frac{1}{(t' - s')}} + C\sqrt{|s - s'| \log\frac{1}{|s - s'|}}\\
\leq (4C\sqrt{\epsilon} + c^*\sqrt{(1+2\epsilon)(1 - \log(1 - 2\epsilon))})\sqrt{(t - s)\log\frac{1}{t - s}}.
\end{gather*}
By making $\epsilon > 0$ small, the first factor on the right can be chosen arbitrarily close to~$c$. This completes the proof of the theorem. 

\end{proof}

\subsection{Computability of minimum and maximum}

Since a sample path~$B$ is almost surely continuous, it almost surely reaches a maximum and a minimum on any given interval. As it turns out, these extremal values can be computed layerwise in~$B$. 

\begin{proposition}\label{proposition:max-layerwise-computable}
The function
\[
\max(B,x,y) = \max \{B(t) \mid t \in [x,y]\}
\]
is computable uniformly in~$x,y$ and layerwise in~$B$. The same is true for the minimum function. 

\end{proposition} 
\begin{proof}
To compute the maximum of $B(t)$ on $[x, y]$ to within $\epsilon$, we run the following simple algorithm: Pick $h_0$ small enough so that $B(t)$ obeys a modulus of continuity with constant $c = 2$ (see Proposition~\ref{prop:modulus-of-continuity-2}) and so that $2\sqrt{h_0\log(1/h_0)} < \epsilon$. Then we know that the maximum of the values $B(r_1), B(r_1 + h_0), B(r_1 + 2h_0), ... , B(r_2)$ must be within $2\sqrt{h_0\log(1/h_0)} $, and therefore within $\epsilon$, of the maximum value of $B(t)$ on $[x,y]$. The minima are also layerwise computable by the same argument.

Note that this argument does not establish the layerwise computability of the time(s) at which the maximum occurs; the best we can say using this argument is that the time(s) are $\Pi_1^0$ in $B$, and the argument uses the randomness deficiency of $B$ and so is not uniform. 

\end{proof}

\begin{proposition}\label{prop:max-not-computable}
Local maxima and local minima of a \ML random Brownian motion are Martin-L\"of random reals (in particular, they cannot be computable reals).
\end{proposition}
\begin{proof}

Fix two rational numbers $x<y$. It is known classically that $\max(\B,0,y)$ is distributed according to the density function
\[
f(a) = 2\cdot \frac{e^{-a^2/(2y)}}{\sqrt{2\pi y}} 
\]
for $a \geq 0$, and $f(a) = 0$ for $a<0$ (see~\cite[Theorem 2.21]{MP}). By the Markov property, $\max(\B,x,y)$ has the same distribution as $\B(x)+\max(\B,0,y-x)$, and thus is distributed according to the density function 
\[
g(a) =   \frac{e^{-a^2/(2x)}}{\sqrt{2\pi x}} + 2 \frac{e^{-a^2/(2(y-x))}}{\sqrt{2\pi (y-x)}} 
\]
for $a \geq 0$, and $f(a) = 0$ for $a<0$. It is known that if a computable measure~$\mu$ on~$\R$ admits a continuous positive density function, then its random elements are exactly the Martin-L\"of random reals (see~\cite{HRAppofEffProb}). Since the function
\[
B \mapsto \max(B,x,y)
\]
is layerwise computable, its image measure is computable, and by the above has a continuous positive density function. Moreover, by the randomness preservation theorem since the function
\[
B \mapsto \max(B,x,y)
\]
is layerwise computable, the image of an ML random~$B$ is random for the image measure, hence is Martin-L\"of random for the uniform measure. 

\end{proof}

\begin{corollary}\label{cor:zero-eq-crossing}
If a \ML random~$B$ has a zero on some interval~$[a,b]$, there are $x,y \in [a,b]$ such that $f(x)>0$ and $f(y)<0$.
\end{corollary}

\begin{proof}
Otherwise~$0$ would be a local maximum or minimum, which would contradict Proposition~\ref{prop:max-not-computable}.
\end{proof}


\section{Zero sets of \ML random Brownian motion}

In this section, we study the properties of the zero set 
\[
Z_B = \{t\geq 0: B(t) = 0\} 
\]
of \ML random paths. Once again, we will need some classical results to prove our effective theorems. Most importantly, we will need the next proposition, which gives an exact expression of the probability that a path has a zero in a given interval. 

\begin{proposition}[see \cite{PeresInvitation}]\label{prop:hitting-prob-1}
For any $\epsilon \in (0, 1)$ and $a>0$ 
$$
\Pr_0 \Big(B(s) = 0 \text{ for some }s \in [a, a + \epsilon]\Big) = \frac{2}{\pi} \arctan\left( \sqrt{\frac{\epsilon}{a}} \right)
$$
which is $\sim$ $\frac{2}{\pi} \sqrt{\frac{\epsilon}{a}}$ as $\epsilon$ tends to $0$. 
\end{proposition}

We shall also need the following lemma. 

\begin{lemma}\label{lem:bigger-x-smaller-prob}
Let $[a,b]$ be a sub-interval of $[0,\infty)$. Then for all~$x$
\[
\Pr_0 \big(\B \text{ has a zero in } [a,b] \big) \geq \Pr_x \big( \B \text{ has a zero in } [a,b] \big)
\]
\end{lemma}

\begin{proof}
Consider the random variable~$\B$ consisting of a Brownian motion starting at~$0$, and form the variable $\B'$ defined as follows: $\B'(t)=x-\B(t)$ for $t \leq \tau$ and $\B'(t)=\B(t)$ for $t \geq \tau$, where~$\tau$ is the first time~$s$ at which $\B(s)=x-\B(s)$. Then the distribution of $\B'$ is that of a Brownian motion starting at~$x$. Moreover, if $\B'(t)=0$ for some $t \in [a,b]$, then by continuity we have $\tau < t$, and thus $\B(t)=\B'(t)=0$. This shows that 
\[
\Pr (\B' \text{ has a zero in } [a,b] \big) \leq \Pr \big( \B \text{ has a zero in } [a,b] \big)
\]
and the result follows. 

\end{proof}

\subsection{The zero set of~$B$ is layerwise recursive in~$B$}

Following~\cite[Definition 5.1.1]{Weihrauch2000}, we say that a closed set~$\mathcal{C}$ is \emph{recursive} if the predicate
\[
\mathcal{C} \cap (a,b) = \emptyset
\]
over a pair $(a,b)$ of rationals, is decidable. 

\begin{remark}\label{rem:minimum}
Note that a recursive closed set is in particular a $\Pi^0_1$ class. Not all $\Pi^0_1$ classes  are recursive. For example, the minimum element of a bounded recursive closed set is necessarily a computable real, a property that not all bounded $\Pi^0_1$ subsets of~$\R$ have. To see this, suppose without loss of generality that all members of $\mathcal{C}$ are positive. Then the minimum is lower semicomputable as 
\[
\min(\mathcal{C}) = \sup \{q \in \Q \mid (0,q) \cap \mathcal{C}=\emptyset\}
\]
and upper semicomputable as
\[
\min(\mathcal{C}) = \inf \{q \in \Q  \mid \exists q' \in \Q ~(q',q) \cap \mathcal{C} \not=\emptyset\}
\]
\end{remark}

The main result of this subsection is that the zero set $Z_B$ is recursive layerwise in~$B$. To prove this fact, we first need to show the following proposition.

\begin{proposition}\label{prop:origin-not-isolated}
For $B$ \ML random, the origin is not an isolated zero.  
\end{proposition}

\begin{proof}

For all~$k$, we know from Proposition~\ref{prop:hitting-prob-1} that the probability for  Brownian motion not having a zero on the interval $(2^{-3k},2^{-3k}+2^{-k})$ is 
\[
1-\frac{2}{\pi} \arctan(2^k)
\]
which limits to zero, computably, as $n \to \infty$. Moreover, we argued above that not having a zero in a given rational interval is a $\Sigma^0_1$ event, thus this gives us a Martin-L\"of test (in fact, a Schnorr test), and thus a Martin-L\"of random~$B$ must have a zero in infinitely many intervals of type $(2^{-3k}, 2^{-3k}+2^{-k})$.

\end{proof}

\begin{proposition}\label{prop:no-computable-zero}
For $B$ Martin-L\"of random, the set $Z_B$ does not contain any computable real other than~$0$. 
\end{proposition}

\begin{proof}
Suppose $x>0$ is computable. Let $[a_k,a_k+2^{-k}]$ be a computable sequence of rational intervals containing~$x$. The probability for~$\B$ to have a zero in $[a_k,a_k+2^{-k}]$ is $O(2^{-k/2})$ (the multiplicative constant depending on~$x$) and by Corollary~\ref{prop:max-not-computable}, having a zero in $[a_k,a_k+2^{-k}]$ for a \ML random Brownian motion is equivalent to having a positive and a negative value on $[a_k,a_k+2^{-k}]$, which is a $\Sigma^0_1$ property. Therefore, this induces a Martin-L\"of test, and thus any \ML random~$B$ must have no zero in $[a_k,a_k+2^{-k}]$ for some~$k$. 
\end{proof}

\begin{theorem}\label{thm:ZB-layerwise-recurive}
For $B$ a \ML random path, $Z_B$ is a non-empty closed set which is recursive layerwise in~$B$.
\end{theorem}
\begin{proof}
$Z_B$ is closed because $B(t)$ is continuous. 

Let us now prove that $Z_B$ is decidable layerwise in~$B$. We need to see how to decide, layerwise in~$B$, whether~$B$ has a zero in a rational interval~$(a,b)$ with $a<b$. If~$a=0$, we know by Proposition~\ref{prop:origin-not-isolated} that answer is necessarily yes, so we can assume $a>0$. The first important observation is that, in case~$B$ does have a zero on $(a,b)$, it must take a positive and a negative value somewhere on the interval. Otherwise, $0$ would be a local maximum or minimum, which by Proposition~\ref{prop:max-not-computable} cannot happen. Conversely, having a positive and a negative value on the interval guarantees the existence of zero. Since having a positive and a negative value is a $\Sigma^0_1$ event, the predicate $\mathcal{C} \cap (a,b) \not= \emptyset$ is itself $\Sigma^0_1$, uniformly in~$B$. It remains to show that $\mathcal{C} \cap (a,b) = \emptyset$ is $\Sigma^0_1$ layerwise in~$B$. Note that by Proposition~\ref{prop:no-computable-zero}, $B$ cannot have a zero at~$a$ nor~$b$, so 
\[
\mathcal{C} \cap (a,b) = \emptyset \Leftrightarrow \max(B,a,b) > 0 ~\text{or}~  \min(B,a,b) < 0
\]
Since $\max(B,a,b)$ and $\min(B,a,b)$ are layerwise computable in~$B$, this shows that $\mathcal{C} \cap (a,b) = \emptyset$ is a $\Sigma^0_1$ predicate. 

\end{proof}

This theorem yields several useful corollaries.

\begin{corollary}\label{cor:first-zero-computable}
The first zero of~$B$ on an interval $[a,b]$ with $a<b$ rationals (taking value $\bot$ if there is no such zero) is computable layerwise in~$B$ and uniformly in $a,b$. 
\end{corollary}

\begin{proof}
Again, note that  if $a=0$, then the first zero is $0$. Now, suppose $a>0$. By the Proposition~\ref{prop:no-computable-zero}, $B$ cannot have a zero at~$a$ nor at $b$, thus $Z_B \cap [a,b]= Z_B \cap (a,b)$, and one can immediately check whether the latter is empty (layerwise in~$B$) since $Z_B$ is recursive layerwise in~$B$. In case $Z_B \cap [a,b] \not= \emptyset$, and we have explained in Remark~\ref{rem:minimum} that the minimum of a recursive closed set can be computed (uniformly in a code for this closed set). It is easy to see that $Z_B \cap [a,b]$ is itself recursive layerwise in~$B$ and uniformly in $a,b$, thus its minimum element can be computed layerwise in~$B$ and uniformly in $a,b$. 
\end{proof}

\begin{corollary}\label{cor:finite-uninion-intervals}
If $F$ be is a finite union of rational intervals, $\Pr \{Z_\B \cap \mathcal{F} \not = \emptyset\}$ is computable uniformly in a code for~$F$. If $U$ is an effectively open subset of $[0,1]$, then $\Pr \{Z_\B \cap \mathcal{U} \not = \emptyset\}$ is lower semi-computable uniformly in an index for~$U$. 
\end{corollary}

\begin{proof}
For a given~$F$, let $\mathcal{E}_F$ be the event $[Z_\B \cap F \not= \emptyset]$. By Theorem~\ref{thm:ZB-layerwise-recurive}, the characteristic function $\mathbf{1}_{\mathcal{E}_F}$ is layerwise computable, uniformly in a code for~$F$. Thus, by Theorem~\ref{thm:layerwise-integral}
\[
\Pr [Z_\B \cap F \not= \emptyset] = \int_B \mathbf{1}_{\mathcal{E}_F}(B)\, d\Pr(B)
\]
is computable uniformly in a code for~$F$. To get the lower semi-computability of $\Pr \{Z_\B \cap \mathcal{U} \not = \emptyset\}$ when $\mathcal{U}$ is an effectively open set, it suffices to observe that 
\[
\Pr [Z_\B \cap U \not= \emptyset] = \sup_t \Pr [Z_\B \cap U[t] \not= \emptyset]
\]
where $U[t]$ is the approximation of~$U$ at stage~$t$, which is a finite union of rational intervals. 
\end{proof}

Finally, we show that $Z_B$ has no isolated point for~$B$ \ML random.

\begin{proposition}
For $B$ \ML random, $Z_B$ has no isolated point. 
\end{proposition}

\begin{proof}

Consider $\tau_q = \text{inf} \{ t \geq q : B(t) = 0\}$, the first zero after some $q \in \mathbb{Q}$. By closure of $Z_B$, the infimum is a minimum. Moreover, $\tau_q$ is layerwise computable in $B$ by Corollary~\ref{cor:first-zero-computable}  and is an almost surely finite stopping time. Thus by the constructive strong Markov property, $\tau_q$ is not an isolated zero from the right. 

Now, consider zeros that are not of the form $\tau_q$. Call some such zero $t_0$. To see it is not isolated from the left, consider a sequence of rationals $q_n \uparrow t_0$. By assumption on~ $t_0$, for all $n$ there is some $\tau_{q_n} \in (q_n, t_0)$, so $t_0$ is not an isolated zero from the left. 
\end{proof}

\subsection{Effective version of Kahane's Theorem}

Next, we prove an effective version of the following theorem of Kahane's, which we will need in the next section.

\begin{theorem}[Kahane]\label{thm:kahane}
Let $E_1$ and $E_2$ be two (disjoint) closed subsets of $[0,1]$ such that $\dim(E_1 \times E_2) >1/2$ then:
\[
\Pr(B[E_1] \cap B[E_2] \not= \emptyset)>0
\]
\end{theorem}

(where $B[E]$ is the set $\{B(t): t \in E\}$ and $\dim$ denotes Hausdorff dimension). We shall prove the following. 

\begin{theorem}\label{thm:effective-kahane}
Let $E_1$ and $E_2$ be two (disjoint) $\Pi^0_1$ classes such that $\dim(E_1 \times E_2) >1/2$ then:
\begin{itemize}
\item[(i)] There exists a Martin-L\"of random path $B$ such that $B[E_1] \cap B[E_2] \not= \emptyset$
\item[(ii)] Given a fixed Martin-L\"of random path $B$, there exists an integer $c$ such that $B[E_1/c] \cap B[E_2/c] \not= \emptyset$
\end{itemize}
\end{theorem}

\begin{proof}
First of all, observe that item (i) of the theorem follows from item (ii). Indeed, if we have a ML random path $B$ and an integer~$c$ such that $B[E_1/c] \cap B[E_2/c] \not= \emptyset$, by the scaling property, $\frac{1}{\sqrt{c}} B(ct)$ is also Martin-L\"of random and satisfies~(i). Thus we only need to prove~(ii). For this we will use the classical version of theorem (Kahane's) theorem, together with Blumenthal's 0-1 law and  some recent results of algorithmic randomness. Recall that Blumenthal's 0-1 law states that any event which only depends on a infinitesimal time interval on the right of the origin (formally, any event in the $\sigma$-algebra $\bigcap_{s>0} \sigma \{B(t): 0\leq t \leq s\}$) has probability either zero or one (see~\cite[Theorem 2.7]{MP}). 

Consider the scaling map $S:B \mapsto \frac{1}{2}B(4t)$. As we saw in Subsection~\ref{subsec:scaling}, $S$ is computable and preserves Wiener measure $\Pr$ on $C[0,1]$. Moreover, this map is \emph{ergodic}. Indeed, let $\mathcal{A}$ be an $\Pr$-measurable event which is invariant under~$S$, i.e we have $B \in \mathcal{A} \Leftrightarrow S(B) \in \mathcal{A}$. By induction,  $B \in \mathcal{A} \Leftrightarrow \forall n\, S^n(B) \in \mathcal{A}$. The function $S^n(B)$ on $[0,1]$ only depends on the values of $B$ on~$[0,4^{-n}]$. Therefore the event $\mathcal{A}$, which is equal to $[\forall n\, S^n(B) \in \mathcal{A}]$, only depends on the germ of~$B$. By Blumenthal's 0-1 law, this ensures that $\mathcal{A}$ has probability $0$ or $1$. Thus $S$ is ergodic. 

Now, consider the set 
\[
\mathcal{U} = \{B \mid B[E_1] \cap B[E_2] = \emptyset\}
\]
We claim that $\mathcal{U}$ is a $\Sigma^0_1$ subset of $\C([0,1])$. This is because of a classical result in computable analysis: the image of a $\Pi^0_1$ class by a computable function is a $\Pi^0_1$ class. This fact is uniform: from an index of a $\Pi^0_1$ class $P$ and a computable function~$f$ on can effectively compute the index of the $\Pi^0_1$ class $f[P]$.  By uniform relativization, there is a computable function $\gamma$ s.t. given a pair $(f,P)$ where $f$ is a continuous function given as oracle, and $P$ is a $\Pi^0_1$ class of index~$e$, $\gamma(e)$ is an index for $f[P]$ as a $\Pi^{0,f}_1$-class. 
Here we have two $\Pi^0_1$ classes $E_1$ and $E_2$, say of respective indices $e_1$ and $e_2$. By the above discussion $B[E_1]$ and $B[E_2]$ have respective indices $\gamma(e_1)$ and $\gamma(e_2)$ as $\Pi^{0,B}_1$-classes and since the intersection of two $\Pi^0_1$ classes is index-computable, there is a computable function $\theta$ such that $B[E_1] \cap B[E_2]$ has index $\theta(e_1,e_2)$ as a $\Pi^{0,B}_1$-class. Since one can computably enumerate, uniformly in the oracle~$B$, the indices of $\Pi^{0,B}_1$-classes, it follows that the set $\mathcal{U}$ is $\Sigma^0_1$, as wanted. 

We can now apply the effective ergodic theorem proven in~\cite{BienvenuDHMS2012,FranklinGMN2011}: since~$\mathcal{U}$ has measure less than $1$ (by Kahane's theorem) and is a $\Sigma^0_1$ set, there are infinitely many~$n$ such that $S^n(B) \notin \mathcal{U}$ (in fact, the set of such $n$'s is a subset of~$\N$ of positive density), i.e., such that $B[E_1/2^n] \cap B[E_2/2^n] \not= \emptyset$.

\end{proof}


\section{The effective dimension of zeros}

Effective Hausdorff dimension is a modification of Hausdorff dimension for the computability setting. Intuitively, effective Hausdorff dimension describes how ``computably locatable" a point or set is in addition to its size. For example, an algorithmically random point in $\R^n$ has effective Hausdorff dimension $n$ because it can't be computably located any more precisely than a small computable ball, which has Hausdorff dimension $n$. 

There are many equivalent definitions of effective Hausdorff dimension, but we will use the following definition of Mayordomo\cite{MayordomoEffDim}. See the book by Downey and Hirschfeldt \cite{DH}, or papers by Lutz \cite{LutzEffFractaldim} and Reimann \cite{ReimannPhD,Reimann} for more details. 

\begin{definition}
The \emph{effective Hausdorff dimension} of $X \in 2^{\omega}$ is
$$
\cdim(x) \coloneqq \liminf_n \frac{K(X \uh n)}{n}
$$
\end{definition}
This definition can be extended to real numbers by identifying them with their binary representation. \\

In this section, we will try to characterize the effective dimension of the zeroes of \ML random paths. This can be broken down in two questions: 
\begin{enumerate}
\item Given a \ML random~$B$, what is the set $\{\cdim(x) \mid x>0 ~ \text{and} ~  x \in Z_B\}$?
\item Given a real~$x$, can we give a necessary or sufficient condition in terms of the effective dimension of~$x$ for the existence of some \ML random path which has a zero at~$x$? 
\end{enumerate}

As to the first question, Kjos-Hanssen and Nerode~\cite{KHN} have showed that with probability~$1$ over~$B$, $\{\cdim(x) \mid x>0 ~ \text{and} ~  x \in Z_B\}$ is dense in $[1/2,1]$\footnote{this is actually a stronger form of the theorem proven in~\cite{KHN}, but the proof of the latter can easily be adapted}. We make this more precise by showing that for every \ML random path~$B$ (not just almost all paths) $\{\cdim(x) \mid x>0 ~ \text{and} ~  x \in Z_B\}$ is contained in $[1/2,1]$ and contains all the computable reals $>1/2$ of this interval. 

We will answer the second question by proving that having effective dimension at least $1/2$ is necessary, while having effective strictly greater than $1/2$ is sufficient (but not having dimension $1/2$). 

\subsection{The dimension spectrum of $Z_B$}

The next theorem is a direct consequence of the effective version of Kahane's theorem. 

\begin{theorem}
Given a \ML random~path $B$ and computable real $\alpha>1/2$, there exists a real $x$ in~$Z_B$ of constructive dimension~$\alpha$. 
\end{theorem}

\begin{proof}
Let $B$ be such a path and~$\alpha$ such a real. Consider the Bernoulli measure $\mu_p$ (i.e., measure where each bit has probability~$p$ of being a zero, independently of all other bits) such that $p<1/2$ and $-p \log p - (1-p)\log (1-p)=\alpha$. Since $\alpha$ is computable, so is~$p$ (and hence~$\mu_p$), because the function $x \mapsto -x \log x - (1-x)\log (1-x)$ is computable and increasing on $[0,1/2]$. Let 
$E_1=\{0\}$ and $E_2$ be the complement of the first level of the universal Martin-L\"of test for $\mu_p$ (it is a $\Pi^0_1$ class since $\mu_p$ is computable). It is well-known that every set of positive $\mu_p$-measure has Hausdorff dimension $\geq \alpha$, and moreover that every $\mu_p$ random real has constructive Hausdorff dimension~$\alpha$ (see for example Reimann~\cite{ReimannPhD}). Applying Theorem~\ref{thm:effective-kahane}, there exists some $c$ such that $B[E_1/2^c] \cap B[E_2/2^c] \not= \emptyset$. That is, there is some $x \in E_2$ such that $B(2^c x)=0$. Multiplying by $2^c$ just adds $c$ zeros in the binary expansion of~$x$, thus $2^c x$ has the same constructive dimension as $x$, which is $\alpha$. 
\end{proof}

\begin{question}
The previous theorem could be strengthened with some additional effort to $\textbf{0}'$-computable $\alpha$. However, we conjecture that a stronger result is true, namely that for every \ML random~$B$, it holds that 
\[
\{\cdim(x) \mid x>0 ~ \text{and} ~  x \in Z_B\} = [1/2,1]
\]
We do not know how to show this and leave it as an open question. 
\end{question}

\subsection{Being a zero of an \ML random path}

We now address the second of the two above questions: what properties (in terms of effective dimension or Kolmogorov complexity) characterize the reals that belong to $Z_B$ for some \ML random $B$? To do so, we largely borrow from the work of Kjos-Hanssen~\cite{KjosHanssen2009}, but with a number of necessary adaptations to Brownian motion (the paper~\cite{KjosHanssen2009} studies a different stochastic process, namely random closed sets, a particular type of percolation limit sets). Proposition~\ref{prop:hitting-prob-1} gives us a precise expression for the probability of a Brownian motion~$\B$ to have a zero in a given interval. The key step needed to adapt Kjos-Hanssen's techniques is to estimate the probability for $\B$ to have a zero in each of \emph{two} intervals of the same length. 

\begin{proposition}\label{prop:hitting-prob-2}
Let $0<a<b<1$ and $\epsilon >0$. Suppose that the intervals $[a,a+\epsilon]$ and $[b,b+\epsilon]$ are disjoint. Let $\delta$ be the distance between them (i.e., $\delta=b-a-\epsilon$). Let $\A_1$ be the event ``$\B(s) = 0 \text{ for some }s_1 \in [a, a + \epsilon]$" and $\A_2$ be ``$\B(s) = 0 \text{ for some }s_2 \in [b, b + \epsilon]$". Then
\[
\Pr_0 \left(\A_1 \wedge \A_2 \right) \leq  \frac{\epsilon \cdot O(1)}{\sqrt{a \delta}}
\]
where the term $O(1)$ is a constant independent of $a,b,\epsilon$. 
\end{proposition}

\begin{proof}
In this proof, we make use of the following notation: given an event $\A$, $\shift{\A}{\tau}$ the unique (by assumption on $\A$) event such that $t \mapsto B(t+s) \in \shift{\A}{\tau}$ if and only if $t \mapsto B(t) \in \A$.\\

Now, let $\A_1$ and $\A_2$ be the above events, and let us write
\[
\Pr_0 \left(\A_1 \wedge \A_2 \right) = \Pr_0 (\A_1)\Pr_0(\A_2 \mid \A_1)
\]
The term $\Pr_0 (\A_1)$ is, by Proposition~\ref{prop:hitting-prob-1}, equal to $O(\sqrt{\frac{\epsilon}{a}})$. It remains to evaluate the term $\Pr(\A_2 \mid \A_1)$. The event $\A_2$ only depends on the values of~$\B$ on the interval $[b,b+\epsilon]$, thus 
\[
\Pr_0(\A_2 \mid \A_1)= \int_{z \in \R} \Pr_{z} (\shift{\A_2}{(a+\epsilon)})\, f(z) \, dz
\]
where $f$ is the density function of $\B(a+\epsilon)$ conditioned by $\A_1$. By shift invariance of the Wiener measure, we observe that in this expression, the term $\Pr_{z} (\shift{\A_2}{(a+\epsilon)})$ is equal to $\Pr_z(\B\; \text{has a zero in}\, [\delta,\delta+\epsilon])$. This is, in turn, always bounded by $\Pr_0(\B\; \text{has a zero in}\, [\delta,\delta+\epsilon])$, by Proposition~\ref{prop:hitting-prob-1}. Thus
\begin{eqnarray*}
\Pr_0(\A_2 \mid \A_1) & = & \int_{z \in \R} \Pr_{z} (\shift{\A_2}{(a+\epsilon)})\, f(z) \, dz\\
 &  \leq & \int_{z \in \R} \Pr_{0} (\shift{\A_2}{(a+\epsilon)})\, f(z) \, dz\\
 & \leq &  \Pr_{0} (\shift{\A_2}{(a+\epsilon)}) \\
 & \leq & \Pr_0(\B\; \text{has a zero in}\, [\delta,\delta+\epsilon]) \\
 & \leq & \frac{2}{\pi} \arctan\left( \sqrt{\frac{\epsilon}{\delta}}\right) \\
 & \leq & \frac{2}{\pi}  \sqrt{\frac{\epsilon}{\delta}}
\end{eqnarray*}

We have thus established the desired result. 

\end{proof}

\subsubsection{A necessary and a sufficient condition}

Our next theorem gives a necessary condition for a point to be a zero of some \ML random path. 

\begin{theorem}\label{thm:effdimhigh}
If $B$ is a \ML random path, then all members of the set $Z_B \setminus \{0\}$ have effective dimension at least $1/2$. 
\end{theorem}

\begin{proof}

Suppose that for a given $B$, we have $B(a)=0$ for some $a$ such that $\cdim(a) <1/2$. We will show that $B$ is not \ML random. 

Let $1/2<\rho<\cdim(a)$. Take also some rational $b$ such that $0<b<a$. By definition of constructive dimension, for all~$n$, there exists a prefix $\sigma$ of $a$ such that $K(\sigma) \leq \rho |\sigma|-n$. 
For all strings~$\sigma$ such that $0.\sigma>b$, let $I_\sigma=[0.\sigma, 0.\sigma+2^{-|\sigma|}]$ and the event 
\[
\mathcal{E}_\sigma: \left[ \B~\text{has a positive and a negative value in $I_\sigma$} \right]
\]

The event $\mathcal{E}_\sigma$ is a $\Sigma^0_1$ subset of $C [0,1]$, uniformly in~$\sigma$ the probability of $\mathcal{E}_\sigma$ is $O(2^{-|\sigma|/2})$ by Proposition~\ref{prop:hitting-prob-1} (the multiplicative constant depending on~$b$). Define

\[
\mathcal{U}_n = \bigcup \big\{\mathcal{E}_\sigma  \mid  K(\sigma) \leq \rho|\sigma|-n \big\}
\]

By assumption, $B$ belongs to almost all~$\mathcal{U}_n$. However, we have 
\begin{eqnarray*}
\Pr (B \in \mathcal{U}_n) & \leq & O(1) \cdot \sum \{2^{-|\sigma|/2} \mid K(\sigma) \leq \rho|\sigma|-n \big\}\\
 & \leq & O(1) \cdot \sum_\sigma 2^{-K(\sigma)-n}\\
 & \leq & O(2^{-n})
\end{eqnarray*}

Thus the $\mathcal{U}_n$ form a Martin-L\"of test, which shows that $B$ is not \ML random. 

\end{proof}

We now prove an (almost) counterpart of Theorem \ref{thm:effdimhigh}:

\begin{theorem}\label{thm:dim-zero}
Let $x \in [0,1]$ be of effective dimension strictly greater than~$1/2$. Then there exists a Martin-L\"of random path~$B$ such that $B(x)=0$. 
\end{theorem}

The proof is much more difficult and involves the notion of $\alpha$-energy. Given a measure~$\mu$ on $\R$ and $\alpha \geq 0$, the \emph{$\alpha$-energy}  of $\mu$ is the quantity
\[
\int \int \frac{d\mu(x)d\mu(y)}{|x-y|^\alpha}
\]
This quantity might be finite or infinite, depending on the value of $\alpha$. We will need the following two lemmas. 

\begin{lemma}\label{lem:frostman2}
Let $\beta>\alpha\geq 0$. If $\mu$ is a measure satisfying  such that $\mu(A) \leq c \cdot |A|^\beta$ for every interval~$A$ (or equivalently, for every dyadic interval) and for some constant~$c$, then $\mu$ has finite $\alpha$-energy.  
\end{lemma}

\begin{proof}
See~\cite{MP}, proof of Theorem 4.32. 
\end {proof}

\begin{lemma}\label{lem:capacity-and-hitting}
Let $\beta>1/2$ and let $\mu$ be a finite Borel measure on $[0,1]$ such that for every dyadic interval $I$, $\mu(I) \leq c \cdot |I|^\beta$ for some fixed constant~$c$ (and thus by the previous lemma $\mu$ has finite $1/2$-energy).  Then there exists a constant~$c'>0$ such that the following holds: for any set $A \subseteq [1/2,1]$ which is  a countable union of closed dyadic intervals
\[
\Pr_0 \big( Z_\B \cap A \not= \emptyset \big)  \geq c' \cdot \mu(A)^2
\]
\end{lemma}

\begin{proof}
It suffices to prove this theorem for a finite number of intervals, and up to splitting them if necessary we can assume that they all have the same length~$2^{-n}$ for some~$n$. Let $I_1,...,I_k$ be those intervals. Define for all~$k$ the random variable $X_k$ by
\[
X_k = \mu(I_k) \cdot 2^{(n/2)} \cdot \indic{\{Z_\B \cap I_k \not= \emptyset\}}
\]
and $Y=\sum_{j=1}^k X_j$. We want to show that $\Pr(Y>0) \geq \frac{\mu(A)^2}{c_0}$ for constant~$c_0$  which does not depend on~$A$, which immediately gives the result (since $Y>0$ is equivalent to $Z_B \cap A \not=\emptyset$). To do so, we will use the Chebychev-Cantelli inequality
\[
\Pr(Y>0) \geq \frac{\mathbb{E}(Y)^2}{\mathbb{E}(Y^2)}
\]
Let us evaluate separately $\mathbb{E}(Y)$ and $\mathbb{E}(Y^2)$. We have
\begin{eqnarray*}
\mathbb{E}(Y) & = & \sum_{j=1}^k \mathbb{E}(X_j)\\
 & \geq & \sum_{j=1}^k 2^{(n/2)} \cdot \mu(I_j) \cdot c_1 \cdot (\sqrt{2^{-n}})\\
 & \geq & c_1 \sum_{j=1}^k  \cdot \mu(I_j) \\
 & \geq & c_1 \cdot \mu(A)
\end{eqnarray*}
for some constant $c_1 \not = 0$, the second inequality coming from Proposition~\ref{prop:hitting-prob-1}. 

Let us now turn to $\mathbb{E}(Y^2)$, which we need to bound by a constant. We have

\[
\mathbb{E}(Y^2)  =  \sum_{\substack{1\leq i \leq k\\ 1 \leq j \leq k}} \mathbb{E}(X_iX_j)
\]
To evaluate this sum, we decompose it into three parts:

\[
\mathbb{E}(Y^2)  =  \sum_{i=1}^k  \mathbb{E}(X_i^2) \ + \ 2\sum_{\substack{1\leq i <j \leq k \\ I_i, I_j ~\text{adjacent}}} \mathbb{E}(X_iX_j)\  + \ 2\sum_{\substack{1\leq i <j \leq k \\ I_i, I_j ~\text{nonadjacent}}} \mathbb{E}(X_iX_j) 
\]

The first part is an easy computation. For all~$i$,
\begin{eqnarray*}
\mathbb{E}(X_i^2) & = & \mu(I_i)^2 \cdot 2^{n} \cdot \Pr{\{Z_\B \cap I_i \not= \emptyset\}}\\
& = & O\big(\mu(I_i)^2 \cdot 2^{n} \cdot 2^{-(n/2)}\big)\\
 & = & O\big(\mu(I_i) \cdot 2^{-\beta n} \cdot 2^{n} \cdot 2^{-(n/2)}\big)\\
 & = & \mu(I_i) \cdot O\big(2^{(1/2-\beta) n}\big)\\
 & = & \mu(I_i) \cdot O(1)
\end{eqnarray*}

(for the third equality, we use the fact that $\mu(I_i) \leq |I_i|^\beta$, and for the fifth one the fact that $\beta>1/2$). 
Thus
\[
\sum_{i=1}^k \mathbb{E}(X_i^2) = \sum_{i=1}^k \mu(I_i) \cdot O(1) = O(1)
\]
For the second part, we use a rough estimate: first notice that
\[
\mathbb{E}(X_iX_j) = \mu(I_i)\cdot\mu(I_j)\cdot 2^n \cdot \Pr{\{Z_\B \cap I_i \not= \emptyset \, \wedge \, Z_\B \cap I_j \not= \emptyset\}}
\]
and for the second part only, we will use the trivial upper bound:
\[
\Pr{\{Z_\B \cap I_i \not= \emptyset \, \wedge \, Z_\B \cap I_j \not= \emptyset\}} \leq \Pr{\{Z_\B \cap I_i \not= \emptyset \}} = O(2^{-n/2})
\]
Combining this with $\mu(I_j) \leq 2^{-\beta n}$, we get:
\[
\mathbb{E}(X_iX_j) = \mu(I_i)\cdot O(2^{(1/2-\beta)n}) = \mu(I_i) \cdot O(1)
\]
Moreover, each interval $I_i$ has at most two adjacent intervals $I_j$. Thus,
\[
\sum_{\substack{1\leq i <j \leq k \\ I_i, I_j ~\text{adjacent}}} \mathbb{E}(X_iX_j)\leq 2 \sum_{i=1}^k \mu(I_i) \cdot O(1) = O(1)
\]
Finally, for the third part, we will use the fact that the $1/2$-energy of $\mu$ is finite. Let us, for a pair of nonadjacent intervals $I_i,I_j$ with $\max(I_i)<\min(I_j)$, denote by $g(i,j)$ the length of the gap between the two, i.e., $g(i,j)= \min(I_j) - \max(I_i)$.
We have 
\begin{equation}
\sum_{\substack{1\leq i <j \leq k \\ I_i, I_j ~\text{nonadjacent}}} \mathbb{E}(X_iX_j)  = \sum_{\substack{1\leq i <j \leq k \\ I_i, I_j ~\text{nonadjacent}}} \mu(I_i)\cdot \mu(I_j) \cdot 2^n \cdot   \Pr{\{Z_\B \cap I_i \not= \emptyset \, \wedge \, Z_\B \cap I_j \not= \emptyset\}}
\end{equation}

 By Proposition~\ref{prop:hitting-prob-2}, 
\begin{equation}
 \Pr{\{Z_\B \cap I_i \not= \emptyset \, \wedge \, Z_\B \cap I_j \not= \emptyset\}} = \frac{2^{-n}\cdot O(1)}{\sqrt{g(i,j)}}
 \end{equation}
 (note that we use the fact that $I_i$ and $I_j$ are contained in $[1/2,1]$, hence $\min(I_i)$ is bounded away from $0$). 
 
Thus, 
\begin{equation}
\sum_{\substack{1\leq i <j \leq k \\ I_i, I_j ~\text{nonadjacent}}} \mathbb{E}(X_iX_j)  = \sum_{\substack{1\leq i <j \leq k \\ I_i, I_j ~\text{nonadjacent}}}\frac{ \mu(I_i)\cdot \mu(I_j)}{\sqrt{g(i,j)}} \cdot O(1)
\end{equation}
Note that, since $I_i$ and $I_j$ are non-adjacent dyadic intervals of length $2^{-n}$, we have $g(i,j)\geq2^{-n}$. Therefore, for two reals $x,y$, if $x \in I_i$ and $y \in I_j$, then $|y-x| \leq 3 g(i,j)$.  By this observation, we have 
\[
\sum_{\substack{1\leq i <j \leq k \\ I_i, I_j ~\text{nonadjacent}}}\frac{ \mu(I_i)\cdot \mu(I_j)}{\sqrt{g(i,j)}} \leq O(1) \cdot \int \int \frac{d\mu(x)d\mu(y)}{|x-y|^{1/2}} \leq O(1)
\]
(the last inequality comes from the fact that the $1/2$-energy of $\mu$ is finite by Lemma~\ref{lem:frostman2}). \\

We have thus established that $\mathbb{E}(Y^2)=O(1)$, which completes the proof.

\end{proof}

Let $KM$ denote the `a priori' Kolmogorov complexity function (see~\cite[Section 6.3.2]{DH}). Recall that $KM(\sigma) = K(\sigma) + O(\log |\sigma|)$, thus in particular~$K$ can be replaced by~$KM$ in the definition of effective dimension. The reason we need $KM$ instead of~$K$ is the following result of Reimann~\cite[Theorem 14]{Reimann2008}, which we will apply in the proof of Theorem~\ref{thm:dim-zero}: Let $z$ be a real such that $KM(z \uh n) \geq \beta n - O(1)$.  Then, there exists a measure~$\mu$ such that $\mu(A) = O(|A|^\beta)$ for all intervals~$A$, and such that $z$ is Martin-L\"of random for the measure~$\mu$.

\begin{proof}[Proof of Theorem~\ref{thm:dim-zero}]
Let $z$ be of dimension $\alpha>1/2$. Let~$\beta$ be a rational such that $1/2<\beta<\alpha$.  Then for almost all~$n$, $KM(z \uh n) \geq \beta n$.  By Reimann's theorem, let~$\mu$ be a measure such that $\mu(A) = O(|A|^\beta)$ for all intervals~$A$, and such that $z$ is Martin-L\"of random for the measure~$\mu$. 

For all~$n$, let $\mathcal{K}_n$ be the complement of the $n$-th level of the universal Martin-L\"of test over~$(C[0,1], \Pr)$ and consider the set
\[
\mathcal{U}_n = \{x \mid \forall B \in \mathcal{K}_n \; B(x) \not= 0\}
\]
We claim that $\mathcal{U}_n$ is $\Sigma^0_1$ uniformly in~$n$, and $\mu(\mathcal{U}_n) = O(2^{-n/2})$. 
To see that it is $\Sigma^0_1$ suppose that $x \in \mathcal{U}_n$, i.e., $B(x) \not=0$ for all~$B \in \mathcal{K}_n$. The set~$\mathcal{K}_n$ being compact (see Section~\ref{sec:background}), the value of $|B(x)|$ for $B \in \mathcal{K}_n$ reaches a positive minimum. Thus there is a rational~$a$ such that $B(x) > a$ for all~$B \in \mathcal{K}_n$. By uniform continuity of the members of~$\mathcal{K}_n$ (ensured by Proposition~\ref{prop:variation-sqrt-h-log-h}), there is a rational closed interval~$I$ containing~$x$ such that $|B(t)|>a/2$ for all $t \in I$ and $B \in \mathcal{K}_n$. Thus $\mathcal{U}_n$ is the union of intervals $(s_1,s_2)$ such that $\min \{B(t) : t \in [s_1,s_2]\} > b$ for some rational~$b$ and all $B \in \mathcal{K}_n$. Moreover, the condition ``$\min \{B(t) : t \in [s_1,s_2]\} > b$ for all $B \in \mathcal{K}_n$" is $\Sigma^0_1$, because the function $B \mapsto \min \{B(t):t \in [s_1,s_2]\}$ is layerwise computable (thus uniformly computable on~$\mathcal{K}_n$), and the minimum of a computable function on an effectively compact set is lower semi-computable uniformly in a code for that set. This shows that $\mathcal{U}_n$ is $\Sigma^0_1$. 

To evaluate $\mu(\mathcal{U}_n)$, let us first observe that by definition of~$\mathcal{U}_n$,
\[
\Pr_0 (Z_\B \cap \mathcal{U}_n) \leq \Pr_0 (\B \in \mathcal{K}_n~\text{and}~Z_\B \cap \mathcal{U}_n) +2^{-n} \leq 2^{-n}
\]
Applying Lemma~\ref{lem:capacity-and-hitting}, it follows that $\mu(\mathcal{U}_n) = O(2^{-n/2})$, as wanted. Since~$z$ is Martin-L\"of random with respect to~$\mu$, it cannot be in all sets $\mathcal{U}_n$, and thus it must be the zero of some Martin-L\"of random path. 

\end{proof}

\subsubsection{The case of points of effective dimension $1/2$}

In the previous section we showed that no point of effective dimension less than~$1/2$ can be the zero of a ML random path, and that every point of dimension greater than $1/2$ is necessarily a zero of some ML random path. This leaves open the question of what happens at effective dimension exactly~$1/2$. While we do not provide a full answer, we show that among points of effective dimension $1/2$, some are zeros of some ML random path, and some are not. 

The next theorem, which strengthens Theorem~\ref{thm:effdimhigh}, gives a necessary condition for a point to be a zero of some ML random path. 

\begin{theorem}\label{thm:brownian-ample-excess}
If $x>0$ is a zero of some ML random path, then
\[
\sum_n 2^{-K(x \uh n) + n/2} < \infty
\]
\end{theorem}

It is interesting to notice the parallel with the so-called `ample excess lemma' (see~\cite[Theorem 6.6.1]{DH}): a real~$x$ is Martin-L\"of random if and only if $
\sum_n 2^{-K(x \uh n) + n} < \infty$.

\begin{proof}
The proof is an adaptation of that of Theorem~\ref{thm:effdimhigh}. First take a rational~$a$ such that $0<a$. We shall prove the lemma for all~$x>a$, which will be enough since~$a$ is arbitrary. 
For each string~$\sigma$ consider, like in Theorem~\ref{thm:effdimhigh}, the interval $I_\sigma=[0.\sigma, 0.\sigma+2^{-|\sigma|}]$ and the event \[
\mathcal{E}_\sigma: \left[ \B~\text{has a positive and a negative value in $I_\sigma$} \right]
\]

Now, consider the function~$\mathbf{t}$ defined on $C[0,1]$ by 
\[
\mathbf{t}(B) = \sum_{\sigma~ s.t.~ a<0.\sigma} 2^{-K(\sigma)+|\sigma|/2} \cdot \mathbf{1}_{\mathcal{E}_\sigma} (B)
\]
The event $\mathcal{E}_\sigma$ is a $\Sigma^0_1$ subset of $C [0,1]$, uniformly in~$\sigma$. Thus the function~$\mathbf{t}$ is lower semi-computable. Moreover, the probability of $\mathcal{E}_\sigma$ is $O(2^{-|\sigma|/2})$ by Proposition~\ref{prop:hitting-prob-1} (the multiplicative constant depending on~$a$). Thus the integral of~$\mathbf{t}$ is bounded, and therefore~$\mathbf{t}$ is an integrable test (see~\cite{Gacs-notes}). Let now~$B$ be a \ML random path and suppose~$B(x)=0$ for some~$x>a$. Then for almost all~$n$, $a<0.(x \uh n)$. Moreover, for every~$n$, $B$ having a zero in $I_{x \uh n}$, it must in fact have a positive and a negative value on that interval (by Proposition~\ref{prop:max-not-computable}). Thus, by definition of~$\mathbf{t}$
\[
\mathbf{t}(B) + O(1) \geq \sum_n 2^{-K(x \uh n)+n/2} 
\]
(the $O(1)$ accounts for the finitely many terms such that $a \geq 0.(x \uh n)$). But since $B$ is Martin-L\"of random and $\mathbf{t}$ is a integrable test, we have $\mathbf{t}(B) < \infty$, which proves our result. 

\end{proof} 

This theorem shows in particular that if~$x$ is the zero of some \ML random path, then $K(x \uh n)-n/2 \rightarrow +\infty$. 

We now give a sufficient condition which actually is very close to our necessary condition.

\begin{proposition}\label{prop:energy-refinement}
Let $f: \N \rightarrow \N$ be a function such that $\sum_n 2^{-f(n)} <\infty$. Let $\mu$ be a Borel measure on~$[0,1]$ such that for every interval~$A$ of length $\leq 2^{-n}$, $\mu(A) \leq 2^{-\alpha n - f(n)}$. Then~$\mu$ has finite $\alpha$-energy. 
\end{proposition}

\begin{proof}
For now, let us fix some~$x$. Define for all~$n$ the interval $I_n$ to be $[x-2^{-n+1},x-2^{-n}] \cap [0,1]$ and $J_n=[x+2^{-n},x+2^{-n+1}] \cap [0,1]$. Then
\begin{eqnarray*}
\int \frac{d\mu(y)}{|x-y|^\alpha} & \leq & \sum_n \int_{y \in I_n} \frac{d\mu(y)}{|x-y|^\alpha} + \sum_n \int_{y \in J_n} \frac{d\mu(y)}{|x-y|^\alpha} \\
 & \leq & \sum_n 2^{\alpha n}  \mu(I_n) + \sum_n 2^{\alpha n}  \mu(J_n)\\
 & \leq & \sum_n 2^{\alpha n} 2^{-\alpha n - f(n)}  + \sum_n 2^{\alpha n}  2^{-\alpha n - f(n)}\\
 & \leq & 2 \cdot \sum_n 2^{-f(n)}\\
 & < & \infty
 \end{eqnarray*}

Therefore, the $\mu$-integral over~$x$ of $\int \frac{d\mu(y)}{|x-y|^\alpha}$ is itself finite, which is what we wanted. 

\end{proof}

\begin{theorem}\label{thm:KM-zero}
Let $f: \N \rightarrow \N$ be a nondecreasing computable function such that $f(n+1) \leq f(n)+1$ for all~$n$, and such that $\sum_n 2^{-f(n)} < \infty$. Let $x$ be a real such that $KM(x \uh n) \geq n/2+f(n)+O(1)$. Then~$x$ is the zero of some \ML random path. 
\end{theorem}

\begin{proof}
Let $f$ be such a function and~$x$ such a real. By a result of Reimann~\cite[Theorem 14]{Reimann2008}, there exists a measure~$\mu$ such that $\mu(A) \leq 2^{-n/2-f(n)+O(1)}$ for all intervals of length $\leq 2^{-n}$ such that~$x$ is Martin-L\"of random with respect to~$\mu$. By Proposition~\ref{prop:energy-refinement}, $\mu$ has finite $1/2$-energy. The rest of the proof is identical to the proof of Theorem~\ref{thm:dim-zero}.
\end{proof}

\begin{theorem}\label{thm:exact-complexity}
Let $0<\alpha<1$ and let $f: \N \rightarrow \N$ be a Lipschitz function such that $f(n)=o(n)$. Then there exists $x \in [0,1]$ such that $K(x \uh n) = \alpha n  + f(n) +O(1)$.
\end{theorem}

\begin{remark}
This theorem was proven by J. Miller (unpublished) in the particular case where~$f=0$. 
\end{remark}

\begin{proof}
Fix a `large enough' integer~$m$, which we will implicitly define during the construction. We will build the sequence~$x$ by blocks of length~$m$. For~$m$ large enough, the empty string has complexity less than $3 \log m$. Suppose we have already constructed a prefix~$\sigma$ of~$x$ such that $|K(\sigma \uh n) - \alpha n  + f(n)| \leq 3 \log m$ for all~$n \leq |\sigma|$ multiple of~$m$. Pick a string~$\tau$ of length~$n$ such that
\[
K(\tau\mid \sigma) \geq m
\]
We then have 
\[
K(\sigma \tau) \geq K(\sigma) + m - 2\log m - O(1)
\]
On the other hand
\[
K(\sigma 0^m) \leq K(\sigma) + 2\log m + O(1)
\]
For each $i \leq m$, consider the ``mixture" between $0^m$ and $\tau$: $\rho_i= (\tau \uh i)0^{m-i}$. Since $\rho_i$ and $\rho_{i+1}$ differ by only one bit in position $\leq m$ from the right, we have $|K(\sigma \rho_i) - K(\sigma\rho_{i+1})|\leq 2\log m +O(1)$. By this `continuity' property, there must be some~$i$ such that $|K(\sigma \rho_i) - \alpha n  - f(n)| \leq 2\log m +O(1)$ (here the $O(1)$ constant depends on~$f$, but not on~$m$). Thus, for~$m$  large enough, we get $|K(\sigma \rho_i) - \alpha n  - f(n)| \leq 3\log m$. Thus, if~$m$ is large enough, we can iterate this argument to build a sequence~$x$ such that   
$|K(x \uh n) - \alpha n  - f(n)| \leq 3 \log m$ for all~$n$ multiple of~$m$. Since $\alpha n + f(n)$ is a Lipschitz function, this is sufficient to ensure $|K(x \uh n) - \alpha n  - f(n)| =O(m)$. 
\end{proof}

We can finally prove the promised theorem.

\begin{theorem}
Among reals of effective dimension $1/2$, some are zeros of some \ML random path, and some are not.
\end{theorem}

\begin{proof}
By Theorem~\ref{thm:exact-complexity}, first consider a real~$x$ such that $K(x \uh n)=n/2+O(1)$. This real has effective dimension $1/2$ and cannot be a zero of a \ML random path (Theorem~\ref{thm:brownian-ample-excess}). 

Applying Theorem~\ref{thm:exact-complexity} again, let $y$ be a real such that $K(y \uh n) = n + 4\log n +O(1)$. Since for every~$\sigma$, $KM(\sigma) \geq K(\sigma) - K(|\sigma|) - O(1) \geq K(\sigma) - 2\log |\sigma| - O(1)$, it follows that $KM(y \uh n) \geq n+2\log n - O(1)$, and thus $y$ is a zero of some \ML random path (Theorem~\ref{thm:KM-zero}). Of course, $y$ has effective dimension $1/2$ as well. 

\end{proof}

This section leaves open the existence of a precise characterization of the reals~$x$ of dimension $1/2$ for which there exists a \ML random path~$B$ such that $B(x)=0$. Short of an exact characterization, it would be interesting to know whether this depends on Kolmogorov complexity alone. By this, we mean the following question.

\begin{question}
 If $K(x \uh n) \leq K(y \uh n) +O(1)$ and $x$ is a zero of some \ML random path, is~$y$ a zero of some \ML random path? Same question with~$KM$ instead of~$K$. 
\end{question}


\section{Planar Brownian Motion}
\subsection{Brownian motion in higher dimensions}
So far we have talked about Brownian motion on $C[0, 1]$ or $C[\R^{\geq 0}]$, but it is also possible to define Brownian motion in higher dimensions. 

  \begin{definition}

If $\B_1, ..., \B_d$ are independent linear Brownian motions started in $x_1, ..., x_d$, then the process $\{\B(t): t \geq 0\}$ given by $\B(t) = (\B_1(t), ..., \B_d(t))$ is \emph{d-dimensional Brownian motion} started in $(x_1, ..., x_d)$. The d-dimensional Brownian motion started at the origin is also called \emph{standard Brownian motion}. One-dimensional Brownian motion is also called \emph{linear}, and two-dimensional Brownian motion is also called \emph{planar Brownian motion}. 
\end{definition}

And similarly, we have 

\begin{theorem}
A function $B(t) = (B_1(t), ..., B_d(t))$ in the space of continuous functions from $[0, \infty)$ to $\R^d$ with Wiener measure is a \ML random path if and only if $B_1(t), ..., B_d(t)$ are mutually \ML random linear Brownian motion.
\end{theorem}
\begin{proof}
This follows immediately from Van Lambalgen's theorem which states that given a computable probability space~$(\mathbb{X},\mu)$, a pair $(A, B)$ is a \ML random element of the product space $(\mathbb{X}, \mu) \times (\mathbb{X}, \mu)$ if and only if $A$ and $B$ are mutually \ML random elements of $(\mathbb{X}, \mu)$. 
\end{proof}

\begin{theorem}\label{thm:B-hits-derandomizing-points}
At any time $t > 0$, for $B(t)$ a planar \ML random path started at (0,0), $B$ is not random relative to any point $(B_x(t), B_y(t))$ on the path, other than the origin.
\end{theorem}
\begin{proof}
For $B(t)$ a standard planar Brownian motion, the probability that $B(t)$ hits an $\epsilon$-ball around a point $(x, y) \not = (0, 0)$, for $\epsilon < |x^2 + y^2|$ is equivalent to the probability that a planar Brownian motion started at radius $ |x^2 + y^2| = R$ hits an $\epsilon$-ball around zero, by radial symmetry of the planar Brownian motion. The radial part of $d$-dimensional Brownian motion is the Bessel process of order $\nu$ where $d = 2\nu + 2$, and is well understood. In the planar case we are concerned with the Bessel process of order zero. 

Let $\tau_{R, \epsilon}$ be the first hitting time of the Bessel process of order zero started at $R$, hitting to $\epsilon$. Using a result of Haman and Matsumoto \cite{HamanaMatsumoto}, we know that $\Pr(\tau_{R, \epsilon} \leq 1)$ is equal to~

\[
 \int_0^1 \frac{R - \epsilon}{\sqrt{2\pi s^3}} e^{-\frac{(R-\epsilon)^2}{2s}} ds - \int_0^1 \frac{R - \epsilon}{\sqrt{2\pi s^3}} e^{-\frac{(R-\epsilon)^2}{2s} }\left[ \int_0^{\infty} \frac{L_{0, R/\epsilon}(x)}{x} e^{-\frac{x(R- \epsilon)}{\epsilon\sqrt{s}}} dx \right] ds 
 \]

where

$$
L_{0, R/\epsilon}(x) = \frac{I_0(Rx/\epsilon)K_0(x) - I_0(x)K_0(R x/\epsilon)}{(K_0(x))^2 + \pi^2(I_0(x))^2}
$$
and 

$$
I_0(x) = \frac{1}{\pi}\int_0^{\pi} e^{x \cos{t}} dt,
$$

$$
K_0(x) = \int_0^{\infty} \frac{\cos{(xt)}}{\sqrt{t^2 + 1}} dt
$$

These functions are computable, because all the component pieces - cosine, square root, exponentiation, multiplication, and division - are computable, and the integral of a computable function is computable. See the book by Weihrauch \cite{Weihrauch2000} for more details.  Moreover, this integral goes to zero as $\epsilon$ goes to zero, which is more easily seen using a classical result of Spitzer~\cite{Spitzer}:
$$
\lim_{\epsilon \to 0} \log\left(\frac{1}{\epsilon}\right)Pr(\tau_{R, \epsilon} \leq 1) = \int_{R^2/2}^{\infty} \frac{e^{-x}}{2x} dx.
$$
As the right hand side is a constant, and $log(\frac{1}{\epsilon}) \to \infty$, we know that $Pr(\tau_{R, \epsilon} \leq 1) \to 0$. 

Thus we have a Schnorr test relative to the point $(x, y)$, so a \ML random path $B(t)$ will only pass through points $(x, y)$ such that the path $B$ (or a code for the path) is not random relative to $(x, y)$, before time 1. The argument is the same for any finite time, not just time 1, so the statement of the theorem holds. 
\end{proof}

\begin{corollary}
For~$B$ a \ML random planar path, the graph of~$B$ has zero area. 
\end{corollary}
\begin{proof}
Only Lebesgue measure zero many points derandomize any particular real, so any \ML random path hits only Lebesgue measure zero many points.
\end{proof}

\begin{corollary}
For any point $(x, y) \not = (0, 0)$, only measure zero many Brownian paths hit $(x, y)$ (Almost surely, Brownian motion does not hit a given point). 
\end{corollary}
\begin{proof}
A real derandomizes only Lebesgue measure zero many reals.
\end{proof}

\begin{corollary}\label{cor:no-computable-pts}
At any time $t > 0$, for $B(t)$ a standard planar \ML random Brownian motion, $B$ does not pass through any computable point. 
\end{corollary}
\begin{proof}
A \ML random path is always random relative to a computable point. 
\end{proof}


\subsection{Dirichlet Problem}

The Dirichlet problem asks the following question: given a domain (i.e., connected open set)~$U \subseteq \R^n$ and a function $\phi$ defined everywhere on the boundary $\partial U$ of $U$, is there a unique, continuous function $u$ such that $u$ is harmonic on the interior of $U$ and $u = \phi$ on $\partial U$? The Dirichlet problem arises whenever one considers notions of potential - for example, the problem may be thought of as finding the temperature of the interior of a heat-conducting region for which the temperature on the boundary is known, or alternatively, finding the electric potential on the interior of a region for which the charge on the boundary is known. 

These physical interpretations of the problem make it clear that there should be a unique solution, and indeed, many ways of finding this unique solution are known. One method of solving the Dirichlet problem which arises from an intuition of heat diffusion in a heat-conducting substance uses the mathematical model of Brownian motion \cite{Kakutani}.

\begin{definition}
Let $U \subset \R^d$ be a domain. We say that $U$ satisfies the Poincar\'{e} cone condition at $x \in \partial U$ if there exists a cone $V$ based at $x$ with opening angle $\alpha > 0$ and $h > 0$ such that $V \cap \B(x, h) \subset U^c$, where $\B(x, h)$ denotes an open ball around $x$ of radius $h$. 
\end{definition}

\begin{theorem}[Kakutani]\label{Kakutani}
Suppose $U \subset \mathbb{R}^d$ is a bounded domain such that every boundary point satisfies the Poincar\'{e} cone condition, and suppose $\phi$ is a continuous function on $\partial U$. Let $\tau(\partial U) = \inf\{t > 0: B(t) \in \partial U\}$, which is an almost surely finite stopping time. Then the function $u:\overline{U} \to \mathbb{R}$ given by 
$$
u(x) = \mathbb{E}_x\left[\phi(B(\tau(\partial U)))\right], \,\,\,\,\,\,\,\, for\,\, x \in \overline{U},
$$
is the unique continuous function harmonic on $U$ with $u(x) = \phi(x)$ for all $x \in \partial U$.
\end{theorem}

By relativizing Corollary~\ref{cor:first-zero-computable}, we can use the layerwise computability of the hitting time of \ML random Brownian motion to a computable line to show that the solution to the Dirichlet problem is computable in the planar case when the boundary is computable and the condition on the boundary is both computable. Of course, we first need to specify what we mean by that. For example, even assuming that the boundary is a curve -- which it might not be, think for example of an open disk with a smaller disk inside removed -- there are several notion of computable curve we can take, see~\cite{RettingerZ2012}. We will take a very general notion of computability (in the case of curve, it is the most general studied in~\cite{RettingerZ2012}): We assume that~$\partial U$ is computable in the sense that there exists a computable sequence $(C_n)_{n \in \N}$ such that for all~$n$, $C_n$ is a finite set of squares in the 2-dimensional grid $2^{-n} \Z \times 2^{-n} \Z$ whose union is connected, contains~$\partial U$, and every point inside this union of squares is at distance at most $2^{-n+2}$ of the boundary. To formalise the fact that the condition $\phi$ is computable, we assume that there is a uniformly computable family $(\phi_n)_{n \in \N}$, where each~$\phi_n$ is a function which assigns a real value to each square~$\mathbf{c}$, in such a way that this value is within $\epsilon(n)$ of the values of $\phi$ on~$\partial U \cap \mathbf{c}$, and the values of two adjacent squares are within $\epsilon(n)$ of each other, $\epsilon$ being a computable function which tends to $0$ computably in~$n$.

\begin{theorem}[Computable Dirichlet Problem]\label{thm:CDP}
Let $U$ be a bounded domain whose boundary $\partial U$ satisfies the Poincar\'{e} cone condition and $\phi$ a condition on the boundary. Assume $\partial U$ and $\phi$ are computable in the sense described above. Then the solution to the Dirichlet problem -  the unique, continuous function $u: \overline{U} \to \mathbb{R}$ harmonic on $U$ such that $u(x) = \phi(x)$ for all $x \in \partial U$ - is computable. 
\end{theorem}

The rest of the section will be devoted to proving this result. The plan is to prove the theorem in two steps:
\begin{itemize}
\item[(i)] First, we prove it in the particular case where $\partial U$ `squared', i.e., is made of a finite number of vertical and horizontal (i.e, parallel to the x-axis or y-axis) segments with rational endpoints, the list being given explicitly. As we will see, in this case, we can apply the results of the previous sections to compute the first time a \ML random path hits the boundary.   
\item[(ii)] Then we extend it to all computable functions~$\gamma$ by approximation. That is, we approximate $\partial U$ by a squared boundary with arbitrary precision and apply Step~1. 
 
\end{itemize}

Let us first see how to apply the results of the previous section to planar Brownian motion.

\begin{lemma}
For $B(t)$ a Martin-L\"{o}f random planar Brownian motion started at a computable point, seeing when $B(t)$ hits the line parallel to either the $x-axis$ or $y-axis$, if the line is computable, is layerwise decidable in $B(t)$. 
\end{lemma}
\begin{proof}
Without loss of generality, say we are looking for the first time $X(t) = \alpha$, for $B(t) = (X(t), Y(t))$, $\alpha$ computable, $B(t)$ started at $q = (q_x, q_y) \in \Q$.  This is equivalent to looking for the first time $X'(t) = X(t) - q_x$, a standard 1-dimensional Brownian motion, crosses $q_x - \alpha$, which has exactly the same proof as Corollary~\ref{cor:first-zero-computable} above. 
\end{proof}

\begin{lemma}\label{lemma:line-segment}
For $B(t)$ a Martin-L\"{o}f random planar Brownian motion started at a computable point, the first time $B(t)$ passes through a vertical or horizontal line segment with computable endpoints is layerwise computable in $B(t)$. 
\end{lemma}
\begin{proof}

To layerwise computably find the first crossing time through the line segment, we run the following algorithm. Let $r_0 = 0$ be the first time considered. The first crossing of $B(t)$ through the line $y = \alpha$ after $r_0$ is layerwise computable in~$B$, call this time $t_1$. If $t_1$ falls within the line segment, we are done. 
 
Assuming $t_1$ crosses the line away from the line segment, we will call the distance from the line segment $\epsilon_1>0$. In order for $B(t)=(X(t), Y(t))$ to hit the line segment after $t_1$, $X(t)$ must change by more than $\epsilon_1$. By Proposition~\ref{prop:modulus-of-continuity-2}, we can find an $h_1$, layerwise in $X(t)$, such that this does not occur in $(t_1, t_1+h_1)$. We choose $r_1 \in (t_1+h_1/2, t_1 + h_1)$ to be any rational time, and then continue the algorithm by finding the next crossing time $t_2>t_1$ through the line $y = \alpha$.

  Because the line segment has computable vertices, $B(t)$ will not cross through the vertex of the line by Corollary~\ref{cor:no-computable-pts}. This tells us that before hitting the line segment, there is a closest value $\epsilon_L>0$ away from the vertex of line segment such that $B(t)$ crosses no closer than $\epsilon_L$ to the vertex. As above, this $\epsilon_L$ is associated with a time $h_L$ within which $X(t)$ will not cross the line segment. As each $\epsilon_n\geq\epsilon_L$,  each $h_n \geq h_L>0$, so we are incrementing our time steps by at least $h_L/2$ at each stage. Therefore we are taking time steps small enough so that we do not miss the first crossing time, but time steps which are always bounded away from 0, so we must eventually find the first crossing time of $B(t)$ through the line segment. 
  
  \end{proof}

We can now prove our theorem in the restricted case of an explicitly given squared boundary.  

\begin{proposition}\label{prop:piecewise-linear}
If $U$ is a planar region such that $\partial U$ is an explicitly given squared boundary and $\phi$ is a computable function on $\partial U$, then the solution to the Dirichlet problem is computable for $U$.
\end{proposition}
\begin{proof}

By Lemma \ref{lemma:line-segment} the first hitting times on each line segment are computable uniformly in starting point $x$ and layerwise in $B$, and $\delta U$ is composed of finitely many line segments with computable endpoints, so the first hitting time $\tau_B(\partial U)$ to the boundary is layerwise computable in $B$, uniformly in the starting point. Since~$\phi$ is computable, $\phi(\tau_B(\partial U))$ is computable uniformly in starting point $x$ and layerwise in~$B$. 

By Theorem~\ref{thm:layerwise-integral}, the expression
$$
u(x) = \mathbb{E}_x\left[\phi(B(\tau_B(\partial U)))\right], \,\,\,\,\,\,\,\, for\,\, x \in \overline{U}
$$
is computable, uniformly in~$x$, and by Kakutani's classical result~\ref{Kakutani}, this is the solution to the Dirichlet problem. 
\end{proof}


Now, all we need to do is extend this last proposition to the general case. 

\begin{proof}[Proof of Theorem~\ref{thm:CDP}]
Let~$u$ be the solution of Dirichlet's problem (we don't know yet it is computable, but we know it exists from the classical theorem) for condition~$\phi$ on $\partial U$. Given a point~$x \in U$, we first compute, for all~$n$, an approximation $C_n$ of $\partial U$ which are squares of $2^{-n} \Z \times 2^{-n} \Z$. Compute the largest set~$Q$ of squares of $2^{-n} \Z \times 2^{-n} \Z$ which (a) contains the square $\mathbf{c}$ which contains~$x$, (b) does not contain any square in~$C_n$ and (c) is $4$-connected, i.e., every square of $Q_n$ should share an edge with another member of~$Q_n$ (unless there is only one square). Call~$V_n$ the interior of the union of the squares in $Q_n$. Observe that $V_n$ must be contained in~$U$, since it contains a point in~$U$, is connected, and is disjoint from $\partial U$ (if it were not contained in~$U$, then $V \setminus \bar{U}$ and $U \cap V$ would be two non-empty open sets partitioning~$V$, contradicting its connectedness). Observe also that each segment of $\partial V_n$ must be the edge of a square $\mathbf{c} \in C_n$, so we can compute a condition $\psi$ on $\partial V_n$ which is equal to $\phi_n(\mathbf{c})$ on the edge of $\phi_n(\mathbf{c})$ (up to smoothing it out around corners to ensure continuity).\\ 

\noindent \textit{Claim}. For every point~$z \in \partial V_n$, $|\psi(z)-u(z)|< O(\epsilon(n)+2^{-n})$. Indeed, let~$\mathbf{c}$ be the member of $C_n$ which has~$z$ on its edge. Every point of $\mathbf{c}$ is at distance at most $2^{-n+2}$ of the boundary, so there is a square $\mathbf{c'}$ at distance $O(2^{-n})$ of $\mathbf{c'}$ which contains some point $z' \in \partial U$, and the value of $\phi_n(\mathbf{c'})$ is within $\epsilon(n)$ of the value of~$u(z')$. Thus,
\begin{eqnarray*}
|\psi(z)-u(z)|  & \leq & |\psi(z)-\phi_n(\mathbf{c'})| + |\phi_n(\mathbf{c'})-u(z')|+|u(z')-u(z)|\\
&  \leq & O(\epsilon(n)) + \epsilon(n) + O(2^{-n}) 
\end{eqnarray*}
(for the last term, we use the fact that $|z'-z| = O(2^{-n})$ and the fact that~$u$ is harmonic, hence Lipschitz), the constants in the $O$-notations not depending on~$n$, $z'$, $z$. To be precise, we need to add the possible error induced by the `smoothing around corners', but it itself is bounded by $O(\epsilon(n)+2^{-n})$ since the $\phi_n$-values of two adjacent segments of $\partial V_n$ are $O(\epsilon(n)+2^{-n})$-close to each other. Thus, applying the restricted case of our theorem (Proposition~\ref{prop:piecewise-linear}) to $\psi$ and $V_n$, we can compute the value $v_n(x)$ of the solution to Dirichlet's problem on $\partial V_n$ with condition~$\psi$. But since $|\psi-u|=|v_n-u|$ is bounded by $O(\epsilon(n)+2^{-n})$ on $\partial V_n$, this implies that $|v_n-u|$ is also bounded by $O(\epsilon(n)+2^{-n})$ on all of $V_n$ (by the maximum principle, since $v_n-u$ is harmonic). Thus, we have effectively obtained an approximation of $u(x)$ with precision $O(2^{-n}+\epsilon(n))$ uniformly in~$n$ and $x$, which means that $u$ is computable.\\
\end{proof}

%
%
%
%

\noindent \textbf{Acknowledgements.} We would like to thank Jason Rute for useful discussions on this topic, and an anonymous MathOverflow contributor for giving us the above elegant proof of Lemma~\ref{lem:bigger-x-smaller-prob}.

\bibliographystyle{plain}
\bibliography{BrownianMotionBib}

\end{document}